\documentclass{amsart}       

\usepackage{amsmath}
\usepackage{latexsym,amssymb,dsfont}
\usepackage{thmtools,mathtools,amsthm}
\usepackage[T1]{fontenc}
\usepackage[utf8]{inputenc}
\usepackage{tikz-cd}
\usepackage{quiver}
\usepackage{url}
\usepackage{extpfeil}
\usepackage{extarrows}

\usepackage[style=numeric,maxnames=10,backend=biber,
useprefix=true,hyperref=true]{biblatex}
\theoremstyle{definition}
\newtheorem{definition}{Definition}[section]

\theoremstyle{theorem}
\newtheorem{theorem}[definition]{Theorem}

\newtheorem{proposition}[definition]{Proposition}
\newtheorem{corollary}[definition]{Corollary}

\usepackage[mathscr]{euscript}

\usepackage{mymacros_snature}

\usepackage[bookmarksdepth=2,pdfencoding=unicode,colorlinks=true]{hyperref}
\hypersetup{allcolors=[rgb]{0.1,0.1,0.4}}
\usepackage{cleveref}

\begin{document}

\title[Generalized Chevalley criteria in simplicial homotopy type theory]{Generalized Chevalley criteria in simplicial homotopy type theory}

\author[Jonathan Weinberger]{Jonathan Weinberger}
\address{Johns Hopkins University, Department of Mathematics, 3400 N Charles St, 21218 Baltimore, MD, USA}
\email{jweinb20@jhu.edu}
\date{\today}




\maketitle

\begin{abstract}
	We provide a generalized treatment of (co)cartesian arrows, fibrations, and functors. Compared to the classical conditions, the endpoint inclusions get replaced by arbitrary shape inclusions. Our framework is Riehl--Shulman's simplicial homotopy type theory which supports the development of synthetic internal $\inftyone$-category theory.
\end{abstract}

\section{Introduction}\label{sec:intro}

We study formal conditions on cells, fibrations, and fibered functors, generalizing from the well-known theory of cocartesian fibrations.\footnote{Our work could be completely dualized for \emph{cartesian} fibrations, using \emph{right} adjoint right inverse adjunctions instead.} We work in the setting of Riehl--Shulman's simplicial homotopy type theory, which provides a framework for synthetic $\inftyone$-category theory, amenable to computer formalization~\cite{KRW-Yon,KudRzk,sHoTT}. Our conditions generalize the \emph{Chevalley conditions} that traditionally have been used to characterize fibrations internally to $2$-categories~\cite{StrYon,GrayFib,LorRieCatFib}. Riehl--Verity in their $\infty$-cosmos theory have extended these to $\infty$-categories~\cite{RV21}, but this is based on traditional, set-theoretic foundations. We work syntactically in a directed version of homotopy type theory (HoTT)~\cite{hottbook,RijIntro}. By its established semantics~\cite{RS17,Shu19,riehl:2024,riehl-ug,Wei-StrExt} this directed type theory has semantics in internal $\infty$-toposes~\cite{RV21,AFfib,BarwickShahFib,rasekh2021cartesian,MWInt,MarCocart}.

Our work makes precise how characterization theorems for cocartesian arrows, fibrations, and functors are formal consequences from their characterizations via left adjoint right inverse conditions. In the extensive studies of cocartesian (and two-sided) fibrations~\cite{RV21} and \cite{BW21,W22-2sCart}, many of the relevant closure properties are also formal consequences of them being defined via Chevalley (or more generally LARI~\cite[Corollary~6.3.8 and Proposition~6.3.1.4]{RV21}) conditions. The related concept of \emph{relative colimits} has been studied by Lurie~\cite[\S4.3.1]{LurHTT}.

This might provide consequences or inspirations for other type-theoretic frameworks in which notions of fibrations based on axiomatically given shapes have been studied~\cite{CCHM2018,OP16,WL19,RH23}.

\subsection*{Acknowledgments}

I am grateful for financial support by the US Army Research Office under MURI Grant W911NF-20-1-0082. Part of the work was supported by the National Science Foundation under Grant No.~DMS-1928930 while I participated in a program supported by the Simons Laufer Mathematical Sciences Institute (SLMath, formerly known as MSRI). The program was held in the summer of 2022 in partnership with the Universidad Nacional Aut\'{o}noma de M\'{e}xico.

This text is a slight extension of Appendix~A of my PhD thesis~\cite{jw-phd}. I am grateful to Ulrik Buchholtz, Emily Riehl, and Thomas Streicher for important discussions, steady guidance, and fruitful collaborations, during the work on my thesis, resp., the time since. Furthermore, I would like to thank Mathieu Anel, Tim Campion, Sina Hazratpour, and Maru Sarazola for interesting discussions and feedback.

\section{Preliminaries}

\subsection{Martin-Löf type theory}


In its foundation, we are working within a dependent type theory, whose basic entities are dependent types or \emph{families} $\Gamma \vdash A$, where
\[ \Gamma \defeq [x_1 : A_1, x_2 : A_2, \ldots, x_n : A_n]\]
is a \emph{context} capturing dependency on free variables. Dependent terms
\[ x_1 : A_1, x_2 : A_2, \ldots, x_n : A_n \vdash a : A \]
are also called \emph{sections} of the family $A$. A dependent type in the empty context $[ \cdot ]$ is just a (constant) type $\cdot \vdash A$. Locality of type theory justifies that we can restrict to working over the empty context.

Martin-Löf type theory comes with $\Sigma$-types and $\Pi$-types. These are both type formers for types $A$ and families $x : A \vdash B(x)$. The $\Sigma$- or \emph{dependent pair} type $\sum_{a:A} B(a)$ consists of pairs $\pair{a}{b}$ with $a:A$ and $b:B(a)$. In the case that $B$ is a constant type, we can identify $\sum_{a:A} B$ with the cartesian product $A \times B$. The $\Pi$- or \emph{dependent function} type $\prod_{a:A} B(a)$ has as terms the sections $a:A \vdash f(a) : B(a)$. if $B$ is a constant type, then $\prod_{a:A} B(a)$ is the same as the ordinary function type $A \to B$.

For any type $A$ and terms $x,y : A$ we have the \emph{identity type} $(x =_A y)$, whose inhabitants $p : (x =_A y)$ can be thought of as \emph{paths} from $x$ to $y$. This notion of \emph{propositional equality} is modeled after Leibniz's law of \emph{identity of discernibles}, which says that equal objects share the same logical properties. But it is also in line with the homotopical interpretation of Martin-Löf type theory~\cite{AW05}. Any $x : A$, by the introduction rule of identity types, gives rise to a canonical self-loop $\refl_x : (x =_A x)$ (introduction rule for identity types). The elimination rule of identity type says that reflexivity inductively generates the family
\[ x : A, y : A \vdash (x =_A y). \]
This principle can be understood as a version of the Yoneda lemma, and in the directed setting of simplicial type theory one can show that directed versions hold as well (for functorial type families). Furthermore, by path induction, any type gives rise to a family of iterated identity types satisfying the $\infty$-groupoid laws. 

\subsection{Univalence axiom}

We assume the presence of sufficiently many universes. Our constructions can make do with a single fixed universe of that hierarchy, notated $\UU$. We furthermore assume Veovodsky's \emph{univalence axiom}, postulating an equivalence between paths in the universe and \emph{weak equivalences} (\aka~bi-invertible maps). One consequence is that we get, for any small type $A:\UU$, an equivalence between the type of $\UU$-small maps into $A$, and the type of \emph{families} $A \to \UU$. Thus, $\UU$-small dependent types in context $A$ are the same as maps with codomain $A$, and a in fact, they can always taken to be of the form $\pr_1 : \sum_{a:A} B(a) \to A$. This is called \emph{fibrant replacement} or \emph{projection equivalence}. It can be seen as a type-theoretic straightening/unstraightening construction~\cite[Theorem~2.5.1]{BW21}. Concretely, weak equivalence between types is defined as follows. A map $f : A \to B$ is a \emph{(weak) equivalence} if and only if the proposition
\[ \isEquiv(f) \defeq \sum_{g : B \to A} (g \circ f =_{A \to A} \id_A ) \times \sum_{h : B \to A} (f \circ h =_{B \to B} \id_B ) \]
is an equivalence. The type of equivalences from $A$ to $B$ is defined as
\[ A \simeq B \defeq \sum_{f:A \to B} \isEquiv(f).\]
If there exists some equivalence from $A$ to $B$, we might abbreviate this by just writing $A \simeq B$.

\subsection{The homotopy theory of types}

This type theory recalled here in a nutshell has a standard interpretation into Kan complexes which are a model for $\infty$-groupoids. The $(\infty,1)$-category of $\infty$-groupoids forms an $\infty$-topos, and in fact, as shown by Shulman~\cite{Shu19} any $\infty$-topos admits a model structure that gives rise to a model of HoTT. Thus HoTT can be seen as a synthetic theory of $\infty$-groupoids \aka~homotopy types. The notions we are about to discuss next are due to Voevodsky. It will be important to distinguish the types that are homotopically trivial, \ie, \emph{contractible}. Given a type $A$, we say that it is contractible if and only if the type
\[ \isContr(A) \defeq \sum_{x:A} \prod_{y:A} (x=_Ay) \]
is inhabited. A contractible type $A$ comes with a \emph{center of contraction} $c_A : A$ and a \emph{contracting homotopy} $H_A : \prod_{y:A} (c_A =_A y)$. Contractible types are equivalent to the point or \emph{terminal type} $\unit$. Another important class of are the \emph{propositions}. A type $A$ is a proposition if and only if the type
\[ \isProp(A) \defeq \prod_{x,y:A} (x=_Ay) \]
is inhabited. One can show that $\isProp(A) \simeq A \to \isContr(A)$. The intuition is that the propositions are exactly those types that describe a \emph{property} rather than (higher) structure. If we know that $A$ is a true proposition, we can canonically give an inhabitant without any choice of higher data involved. For instance, one can show that $\isProp(\isEquiv(f))$, so it is a property of a map being a weak equivalence. Another important insight of Veovodsky's was that being an equivalence is equivalent to all the \emph{fibers} being contractible. Let $f : A \to B$ be a map between types and $b:B$ be an element. Then the \emph{fiber} of $f$ at $b$ is defined as
\[ \fib(f,b) \defeq \sum_{a:A} (f(a) =_B b). \]
Voevodsky showed that
\[ \isEquiv(f) \simeq \prod_{b:B} \isContr\big(\fib(f,b)\big).\]
By univalence, any type family $P : A \to \UU$ can be understood as a \emph{fibration} $B \to A$ (where $B \simeq \sum_{a:A} P(a)$), so an equivalence is exactly a \emph{trivial fibration}, all of whose fibers are contractible. All of these notions have reasonable translations to their expected semantic counterparts, allowing for doing homotopy theory synthetically. Homotopy theory is concerned with the study of homotopy types, which can be modeled as $\infty$-groupoids.

\subsection{Simplicial homotopy type theory}

To be able to capture synthetic $\infty$-categories, we want to augment standard HoTT by two kinds of new structures, both analogous to cubical type theory~\cite{CCHM2018,OP16}. 

\paragraph{Simplicial shapes}
 This is done by adding in a directed, bi-pointed interval, together with all its cartesian powers. From these, one can carve out familiar shapes like the $n$-simplices $\Delta^n$, its boundaries $\partial \Delta^n$, and the horns $\Lambda_k^n$. Importantly for us, these shapes are constructed to be sets, \ie, $0$-types, or come from a separate strict (pre-)type layer. An extensive account of this is~\cite[Subsection~3.2]{RS17}. In our setting, it is convenient (and in accordance with the intended models) to assume the shapes to be fibrant types, too, but we can always strictify them, as needed, \cf~\cite[Section~2.4]{BW21}.

\paragraph{Extension types}
As another gadget, Riehl--Shulman add in \emph{extension types}: for any shape inclusion $\Phi \hookrightarrow \Psi$ and type $\Gamma \vdash A$, we want to fix a \emph{partial} section $a : \prod_{\Phi \times \Gamma} A$. Then, we want to reify all the \emph{judgmental}, \ie, strict extensions of $a$ into the \emph{extension type}
\[ \exten{\Psi}{A}{\Phi}{a},\]
\ie, the types of $\exten{\Psi}{A}{\Phi}{a}$ are sections $b : \prod_{\Psi \times \Gamma} A$ such that
\[ t : \Phi \vdash a(t) \jdeq b(t).\]
We will assume function extensionality both for strict shapes~\cite[Subsection~4.4]{RS17} and homotopical types. This allows us to prove a de-/strictification equivalence between strict extension types and their homotopical analogues, see~\cite[Subsection~2.4]{BW21}.

\subsection{Synthetic \texorpdfstring{$\inftyone$}{(∞,1)}-category theory}

Using the simplicial extensions to our type theory, we can define for any type $A$ and terms $x,y:A$ the \emph{hom type} or \emph{directed arrow type} as
\[ \hom_A(x,y) \defeq (x \to_A y) \defeq \ndexten{\Delta^1}{A}{\partial \Delta^1}{[x,y]}.\]
A type is \emph{Segal} or a \emph{synthetic pre-$\inftyone$-category} if and only if the restriction map $A^{\Delta^2} \to A^{\Lambda_1^2}$ is a weak equivalence. This means, the type $A$ perceives any pair of composable arrows already as the full $2$-simplex, which exactly means that $A$ supports composition of arrows, uniquely up to homotopy. To obtain \emph{synthetic $\infty$-categories}, we have to add the \emph{Rezk completeness} or \emph{local univalence} condition to the Segal type $A$. Rezk completeness says that the canonical comparison $(x =_A y) \to (x \cong_A y)$ (which is defined by path induction, sending $\refl_x$ to $\id_x \defeq \lambda t.x$) is an equivalence. Here, $(x \cong_A y)$ is the type of bi-invertible arrows in the sense of the hom-type defined above.

After introducing these notions, Riehl--Shulman show~\cite[Section~5 and 6]{RS17} that synthetic $\inftyone$-categories behave in many of the expected ways. The theory of functors and natural transformations is particularly nice: any map between (complete) Segal types automatically is a \emph{functor} in that it preserves compositions. The type of natural transformations between a fixed pair of functors can also be defined as a hom type. Riehl--Shulman also develop a comprehensive theory of homotopy coherent adjunctions in this setting~\cite[Section~11]{RS17}, later complemented by fibered and left adjoint right inverse (LARI) adjunctions~\cite[Appendix~B]{BW21}. Bardomiano Mart\'{i}nez~\cite{BM22} has developed (co)limits in this setting.

\subsection{Synthetic fibered \texorpdfstring{$\inftyone$}{(∞,1)}-category theory}

In this type-theoretic setting of synthetic $\inftyone$-category theory, we are particularly interested in reasoning about various kinds of (functorial) families $B \to \UU$, which we can equivalently capture by notions of \emph{fibrations} $E \fibarr B$. In~\cite{RS17}, Riehl--Shulman study discrete covariant fibrations. This has served as the basis to extend the study to cocartesian fibrations~\cite{BW21}, two-sided cartesian fibrations~\cite{W22-2sCart}, Beck--Chevalley and lextensive fibrations~\cite{Wei-IntSum}, and exponentiable fibrations~\cite{BM22}. A fundamental concept for all these is the \emph{dependent} analogue of the directed arrow type: let $B$ be a type with $a,b : B$, and an arrow $u:(a \to_B b)$. We consider a family $P:B \to \UU$ with terms $d:P\,a$ and $e:P\,b$. The type of \emph{dependent arrows} over $u$ from $d$ $e$ is given by 
\[ \dhom_u^P(d,e) \defeq (d \to^P_u e) \defeq \exten{t:\Delta^1}{P(u(t))}{\partial \Delta^1}{[d,e]}.\] 
We often want to have the family $P:B \to \UU$ in consideration to be functorial, meaning any arrow $u:a \to_B b$ induces a functor $u_! : P\,a \to P\,b$.

A standing assumption for the families in question is that they should be \emph{isoinner}, meaning that the base, the total type, and all the fibers should be Rezk types. This is a reasonable baseline for developing notions of fibered categories. A systematic discussion is to be found in~\cite[Section~4]{BW21}.

This framework and our work therein also draws many inspirations from Riehl--Verity's $\infty$-comos theory~\cite{RV21}, which is another approach to synthetic, model-independent formal $\infty$-category theory, based on traditional foundations. 

\section{Relative adjunctions}\label{ssec:reladj}

We provide a brief treatment of \emph{relative adjunctions} in the sense of Ulmer~\cite{UlmDense}, \cf~also~\cite[Exercise~2.11]{loregian_2021}, \cite[Definition~2.6]{MasarykFormal}. This takes up on a suggestion by Emily Riehl to development a more formal account to cocartesian arrows, or more generally, LARI cells in simplicial homotopy type theory after the analogous results in $\infty$-cosmos theory~\cite{RV21} by Riehl--Verity. As a payoff, we will see that the Chevalley condition defining the LARI cells implies the Chevalley condition for LARI fibrations in the sense of~\cite{BW21}, and likewise for LARI functors. 

\begin{definition}[Transposing relative adjunction]
	Let $A,B,C$ be Rezk types and $(g: C \to A \leftarrow B: f)$ a cospan. A \emph{(transposing) left relative adjunction} of $f$ and $g$ consists of a functor $\ell:C \to B$ together with a fibered equivalence
	\[ \big(\comma{\ell}{B} \equiv_{C \times B} \comma{g}{f}\big) \simeq \prod_{\substack{c:C \\ b:B}} \hom_B(\ell\,c, b) \simeq \hom_A(g\,c, f\,b). \]
	Given such data, we call $\ell$ a \emph{(transposing) left adjoint of $f$ relative to $g$} or \emph{(transposing) $g$-left adjoint of $f$}.
\end{definition}

In case $C \jdeq A$ and $g \jdeq \id_A$ one obtains the usual notion of (transposing) adjunction. There also exists a relative analogue of the units. We might occasionally drop the predicate ``left'' in our discussion since we will only consider the left case. But note that relative adjunctions are a genuinely asymmetric notion.

\begin{definition}[Relative adjunction via units]\label{def:reladj-units}
	Let $A,B,C$ be Rezk types and $(g: C \to A \leftarrow B: f)$ a cospan. A \emph{(transposing) left relative adjunction} consists of a functor $\ell:C \to B$ together with a natural transformation $\eta: g \Rightarrow_{C \to A} f\ell$, called \emph{relative unit}, such that the transposition map
	\[ \Theta_\eta \defeq \lambda b,c,k.fk \circ \eta_c : \comma{\ell}{B} \to_{C \times B} \comma{g}{f} \]
	is a fiberwise equivalence.  
\end{definition}

By the characterizations about type-theoretic weak equivalences, \Cref{def:reladj-units} translates to:
\begin{align}\label{eq:reladj-units}
	\prod_{\substack{c:C \\b:B}} \prod_{m:gc \to_A fb} \isContr\Big( \sum_{k:\ell\,c \to_B b} \Theta_\eta(k) = m \Big)
\end{align}
Diagrammatically, this can be depicted as follows, demonstrating once more the generalization from the usual notion of adjunction:
\[\begin{tikzcd}
	{g\,c} && {f\,b} & b \\
	{f\,\ell c} &&& {\ell\,c}
	\arrow["{\forall\,m}", from=1-1, to=1-3]
	\arrow["{\eta_c}"', from=1-1, to=2-1]
	\arrow["{f\,k}"', from=2-1, to=1-3]
	\arrow["{\exists! \,k}", dashed, from=2-4, to=1-4]
\end{tikzcd}\]
It turns out that also in the synthetic setting we recover the equivalence of relative adjunctions with \emph{absolute left lifting diagrams (ALLD)}, whose universal property in terms of pasting diagrams can be (informally or analytically) described as follows.
A lax diagram
\[\begin{tikzcd}
	&& B \\
	C && A
	\arrow[""{name=0, anchor=center, inner sep=0}, "g"', from=2-1, to=2-3]
	\arrow["f", from=1-3, to=2-3]
	\arrow[""{name=1, anchor=center, inner sep=0}, "\ell", from=2-1, to=1-3]
	\arrow["\eta"', shorten <=2pt, shorten >=2pt, Rightarrow, from=0, to=1]
\end{tikzcd}\]
is an absolute lifting diagram if and only if any given lax square on the left factors uniquely as a pasting diagram as demonstrated below left:
\[\begin{tikzcd}
	X && B & {} & X && B \\
	C && A & {} & C && A
	\arrow["\gamma"', from=1-1, to=2-1]
	\arrow["g"', from=2-1, to=2-3]
	\arrow["\beta", from=1-1, to=1-3]
	\arrow["f", from=1-3, to=2-3]
	\arrow["{=}"{description}, draw=none, from=1-4, to=2-4]
	\arrow["\gamma"', from=1-5, to=2-5]
	\arrow[""{name=0, anchor=center, inner sep=0}, "g"', from=2-5, to=2-7]
	\arrow[""{name=1, anchor=center, inner sep=0}, "\beta", from=1-5, to=1-7]
	\arrow["f", from=1-7, to=2-7]
	\arrow["{\forall\,\mu}"{description}, shorten <=10pt, shorten >=10pt, Rightarrow, from=2-1, to=1-3]
	\arrow[""{name=2, anchor=center, inner sep=0}, from=2-5, to=1-7]
	\arrow["{\exists!\,\mu'}", shorten <=2pt, shorten >=2pt, Rightarrow, from=2, to=1]
	\arrow["\eta", shift right=5, shorten <=2pt, shorten >=2pt, Rightarrow, from=0, to=2]
\end{tikzcd}\]
Accordingly we define this type-theoretically\footnote{For a first discussion about lax squares and pasting diagrams in sHoTT~\cf~\cite[Appendix~A]{BW21}. We do currently not have a systematic account to these. \Eg~certainly at some point a pasting theorem \`{a} la~\cite{infty2pasting} would be most desirable. This would presumably require a categorical universe validating a directed univalence principle, and possibly also modalities from cohesion.} as follows:
\begin{definition}[Absolute left lifting diagram]
	A diagram
	\[\begin{tikzcd}
		&& B \\
		C && A
		\arrow[""{name=0, anchor=center, inner sep=0}, "g"', from=2-1, to=2-3]
		\arrow["f", from=1-3, to=2-3]
		\arrow[""{name=1, anchor=center, inner sep=0}, "\ell", from=2-1, to=1-3]
		\arrow["\eta"', shorten <=2pt, shorten >=2pt, Rightarrow, from=0, to=1]
	\end{tikzcd}\]
	is an \emph{absolute lifting diagram (ALLD)} if the following proposition is satisfied:
	\[ \isALLD_{\ell,f,g}(\eta) \defeq \prod_{\substack{X:\UU \\ \beta:X \to B \\ \gamma: X \to C}} \prod_{\mu:g\gamma \Rightarrow f\beta} \isContr \Big(\sum_{\mu': \ell \gamma \Rightarrow \beta} \prod_{x:X} (f\mu_x' \circ \eta_{\gamma\,x} =_{g\gamma\,x \to f\beta\,x} \mu_x)  \Big)\] 
\end{definition}
Note that one can infer the data $\angled{\ell,f,g}$ from $\eta$ alone. We might also speak of $\eta$ as an \emph{absolute left lifting cell}.

Diagrammatically, the demanded identity of morphisms reads:
\[\begin{tikzcd}
	{g(\gamma\,x)} && {f\ell(\gamma\,x)} && {f(\beta\,x)}
	\arrow["{\eta_{\gamma\,x}}", from=1-1, to=1-3]
	\arrow["{f\mu_x'}", from=1-3, to=1-5]
	\arrow[""{name=0, anchor=center, inner sep=0}, "{\mu_x}"', curve={height=24pt}, from=1-1, to=1-5]
	\arrow[shorten >=2pt, Rightarrow, no head, from=1-3, to=0]
\end{tikzcd}\]

The above definitions can be dualized to obtain relative \emph{right} adjoints and absolute \emph{right} lifting diagrams. Because of the inherent asymmetry some analogies to the case of ordinary adjunctions are missing (such as the presence of \emph{both} units and counits). However, we can still provide a characterization result.

\begin{theorem}[Characterizations of relative left adjunctions, \cf~\protect{\cite[Thm.~3.5.8/3]{RV21}}, \protect{\cite[Thm.~11.23]{RS17}}, \protect{\cite[Thm.~B.1.4]{BW21}}]\label{thm:reladj-char}
	Let $A,B,C$ be Rezk types and $(g: C \to A \leftarrow B: f)$ a cospan. Then the following types are equivalent propositions:
	\begin{enumerate}
		\item\label{it:reladj-transp} The type $\sum_{\ell:C \to B} \comma{\ell}{B} \equiv_{C \times B} \comma{g}{f}$ of (transposing) $g$-left adjoints of $f$.
		\item\label{it:reladj-unit} The type $\sum_{\ell:C \to B} \sum_{\eta:g \Rightarrow f\ell} \isEquiv\big( \Theta_\eta \big)$ with $\Theta$ as in~\Cref{def:reladj-units}.
		\item\label{it:reladj-alld} The type $\sum_{\ell:C \to B} \sum_{\eta:g \Rightarrow f\ell} \isALLD_{\ell,f,g}(\eta)$ of completions of the cospan consisting of $f$ and $g$ to an ALLD.
	\end{enumerate}
\end{theorem}
 
\begin{proof}
	In parts we can work analogously as in the proof of \cite[Theorem~11.23]{RS17}. In particular, an equivalence between the types from~\Cref{it:reladj-transp} and~\Cref{it:reladj-unit} follows\footnote{The classical version is due to~\cite[Lemma~2.7]{UlmDense}, and it works by the analogous argument.} just as in \emph{loc.~cit.} by using the (covariant discrete) Yoneda Lemma~\cite[Section~9, and (11.9)]{RS17}. Next, analogously as in the proof of \cite[Theorem~11.23]{RS17} one also shows that, given $\ell:C \to B$, the type $\sum_{\eta:g \Rightarrow f\ell} \isEquiv\big( \Theta_\eta \big)$ is a proposition.
	
	We now fix $\ell:C \to B$ and $\eta:g \Rightarrow f\ell$. Recall~\eqref{eq:reladj-units}. The direction from~\Cref{it:reladj-alld} to~\Cref{it:reladj-unit} follows by setting $X \jdeq \unit$. Conversely, we see that we get from~\Cref{it:reladj-unit} to~\Cref{it:reladj-alld} by ``reindexing'' Condition~\eqref{eq:reladj-units} along any given span $(\gamma : C \leftarrow X \rightarrow B:\beta)$.\footnote{More precisely, we use the fact that, given a family of propositions $P:A \to \Prop$, there is an equivalence $\Phi:\prod_{a:A} P(a) \simeq \prod_{\substack{X:\UU \\ \alpha:X \to A}} \prod_{x:X} P(\alpha \, x):\Psi$. We can take $\Phi(\sigma) \defeq \lambda X,\alpha,x.\sigma(\alpha\,x)$ and $\Psi(\tau) \defeq \lambda a.\tau(\unit)(a)(\pt)$.}
\end{proof}

\begin{corollary}\label{cor:unique-left-rel-adj}
	Given a cospan $(g:C \to A \leftarrow B:f)$, if both $\ell, \ell':C \to B$ are left adjoints to $f$ relative to $g$, then there is an identity $\ell = \ell'$:
	\[\begin{tikzcd}
		&& B \\
		C && A
		\arrow[""{name=0, anchor=center, inner sep=0}, "g"', from=2-1, to=2-3]
		\arrow[""{name=1, anchor=center, inner sep=0}, "f", from=1-3, to=2-3]
		\arrow[""{name=2, anchor=center, inner sep=0}, "\ell"{description}, curve={height=6pt}, from=2-1, to=1-3]
		\arrow[""{name=3, anchor=center, inner sep=0}, "{\ell'}"{description}, curve={height=-12pt}, from=2-1, to=1-3]
		\arrow[shorten <=3pt, shorten >=3pt, Rightarrow, no head, from=2, to=3]
		\arrow["\eta"{description}, shorten <=6pt, shorten >=6pt, Rightarrow, from=0, to=1]
	\end{tikzcd}\]
\end{corollary}

We write $\relAdj{\ell}{g}{f}$ if it exists.

\begin{definition}[Relative LARI adjunction]
	A relative left adjunction is called \emph{relative LARI adjunction} if its relative unit is invertible.
\end{definition}

\section{LARI cells, fibrations, and functors}\label{sec:lari-stuff}

\subsection{LARI cells}\label{ssec:lari-cells}

For this section, we fix the following data. Let $j: \Phi \hookrightarrow \Psi$ be a shape inclusion. Let $B$ be a Rezk type and $P:B \to \UU$ be an isoinner family. For its unstraightening $\pi: E \fibarr B$ the diagram induced by exponentiation is given through: 
\[\begin{tikzcd}[sep=1.5ex]
	{\mathllap{E^\Psi \equiv} \sum_{\pair{u}{f}:\Phi^E} \sum_{v:\ndexten{\Psi}{B}{\Phi}{u}} \exten{\Psi}{v^*P}{\Phi}{f}} &&&& {E^\Phi \simeq \sum_{u:\Phi \to B} \mathrlap{\prod_\Phi u^*P}} \\
	\\
	{\mathllap{B^\Psi \equiv \sum_{u:\Phi \to B}} \ndexten{\Psi}{B}{\Phi}{u}} &&&& {B^\Phi}
	\arrow[from=1-1, to=3-1]
	\arrow[from=3-1, to=3-5]
	\arrow[from=1-1, to=1-5]
	\arrow[from=1-5, to=3-5]
\end{tikzcd}\]
An element $v:\Psi \to B$ is to be understood as \emph{$\Psi$-shaped cell} (or \emph{diagram}) in the type $B$. A section $g:\prod_{t:\Psi} P(v(t))$ is a \emph{dependent $\Psi$-shaped cell (over $v$)} in the family $P$.

\begin{definition}[$j$-LARI cell]\label{def:lari-cell}
	Let $g: \Psi \to E$ be a $\Psi$-shaped cell in $E$, lying over $\angled{u,v,f}$ with $u:\Phi \to B$, $v:\Psi \to B$, and $f:\Phi \to E$ (both the latter lying over $u$). We call $g$ a \emph{$j$-LARI cell} if the ensuing canonical commutative diagram\footnote{By some slight abuse of notation $g$ really stands for the whole tuple $\angled{u,v,f,g}$, and the homotopy is reflexivity. This is a valid reduction due to fibrant replacement.}
	\[ 
	\begin{tikzcd}
		&& {E^\Psi} \\
		\unit && {B^\Psi \times_{B^\Phi} E^\Psi}
		\arrow["g", from=2-1, to=1-3]
		\arrow[""{name=0, anchor=center, inner sep=0}, "{\angled{u,v,f}}"', from=2-1, to=2-3]
		\arrow["{\pi' \mathrlap{\defeq j \cotens \pi}}", from=1-3, to=2-3]
		\arrow[shorten <=15pt, shorten >=15pt, Rightarrow, no head, from=1-3, to=0]
	\end{tikzcd}
	\]
	is an absolute left lifting diagram,~\ie~there is a relative adjunction as encoded by the fibered equivalence
	\[ \comma{g}{E^\Psi} \simeq_{E^\Psi} \comma{\angled{u,v,f}}{\pi'},\]
	or, equivalently by~\Cref{thm:reladj-char}
	\begin{align}\label{eq:lari-cell}
		\isLARICell_{j}^P(g) \defeq \isEquiv\big( \Theta_{\refl} \big).
	\end{align}
\end{definition}
where the transposition map $\Theta_\refl$ simply projects the data of a morphism $\beta$ in $E^\Psi$ onto its part in $E^\Phi \times_{B^\Phi} B^\Psi$.

We can re-express this using the following notion.

\begin{definition}[Pushout product]\label{def:po-prod}
	Let $j: Y \to X$ and $k: T \to S$ each be type maps or shape inclusions. The \emph{Leibniz tensor of $j$ and $k$} (or \emph{pushout product}) is defined as the following cogap map:
	\[\begin{tikzcd}[sep=4ex]
		{Y} & {T} & {Y \times S\bigsqcup_{Y \times T} X \times T} && {Y \times T} && {Y \times S} \\
		{X} & {S} & {X \times S} && {X \times T} && {\cdot} \\
		&&&&&&& {X \times S}
		\arrow["{j}"{name=0, swap}, from=1-1, to=2-1]
		\arrow["{k}"{name=1, swap}, from=1-2, to=2-2, swap]
		\arrow["{j\widehat{\otimes} k}"{name=2}, from=1-3, to=2-3]
		\arrow[from=1-5, to=2-5]
		\arrow[from=2-5, to=2-7]
		\arrow[from=1-5, to=1-7]
		\arrow[from=1-7, to=2-7]
		\arrow[from=2-5, to=3-8, curve={height=12pt}]
		\arrow[from=1-7, to=3-8, curve={height=-12pt}]
		\arrow["{j \widehat{\otimes}k}" description, from=2-7, to=3-8, dashed]
		\arrow["\lrcorner"{very near start, rotate=180}, from=2-7, to=1-5, phantom]
		\arrow[Rightarrow, "{\widehat{\otimes}}" description, from=0, to=1, shorten <=5pt, shorten >=5pt, phantom, no head]
		\arrow[Rightarrow, "{\defeq}" {description,pos=0.25}, from=1, to=2, shorten <=9pt, shorten >=9pt, phantom, no head]
	\end{tikzcd}\]
\end{definition}

In particular, recall from~\cite[Theorem 4.2]{RS17}, the explicit formula for the pushout product of two shape inclusions:

\[\begin{tikzcd}
	{\{t:I\,|\,\varphi\}} & {\{s:J\,|\,\chi\}} & {\{ \langle t,s\rangle : I \times J \, | \, (\varphi \land \zeta) \lor (\psi \land \chi)\}} \\
	{\{t:I\,|\,\psi\}} & {\{s:J\,|\,\zeta\}} & {\{ \langle t,s\rangle : I \times J \, | \, \psi \land \zeta\}} \\
	{}
	\arrow[""{name=0, inner sep=0}, from=1-1, to=2-1, hook]
	\arrow[""{name=1, inner sep=0}, from=1-2, to=2-2, hook]
	\arrow[""{name=2, inner sep=0}, from=1-3, to=2-3, hook]
	\arrow[Rightarrow, "{\widehat{\otimes}}" description, from=0, to=1, shorten <=7pt, shorten >=7pt, phantom, no head]
	\arrow[Rightarrow, "{\defeq}"  {description,pos=0.25}, from=1, to=2, shorten <=13pt, shorten >=13pt, phantom, no head]
\end{tikzcd}\]

Now, Condition~\ref{eq:lari-cell} becomes
\begin{align}\label{eq:lari-cell-explicit} \isLARICell_{j}(g) \equiv \prod_{\substack{\angled{r,w,k,m}:\Psi \to E \\ \alpha:\angled{u,v,f} \to \angled{r,w,k}}} \isContr\left(\exten{\pair{t}{s}:\Delta^1 \times \Psi} {P\big(\alpha_1(t,s)\big)} {b_1 \poprod j}{[\pair{g}{m},\alpha_2]}\right),
\end{align}\label{eq:enough-lari-cell}
where we denote by $b_1:\partial \Delta^1 \hookrightarrow \Delta^1$ the boundary inclusion, and $\alpha \jdeq \angled{\alpha_1, \alpha_2, \alpha_3}$ consists of morphisms:
\begin{align*}
	\alpha_1 & : \hom_{\Phi \to B}(u,r), \\
	\alpha_2 & : \ndexten{(\Delta^1 \times \Psi)}{B}{b_1 \poprod j}{[\pair{v}{w},\alpha_1]}, \\
	\alpha_3 & : \prod_{\pair{t}{s}:\Delta^1 \times \Phi} P\big(\alpha_1\,t\,s \big)
\end{align*}
Intuitively, this means that the given data $\angled{\alpha,g}$ can be uniquely lifted as indicated in~\Cref{fig:lari-cell}.

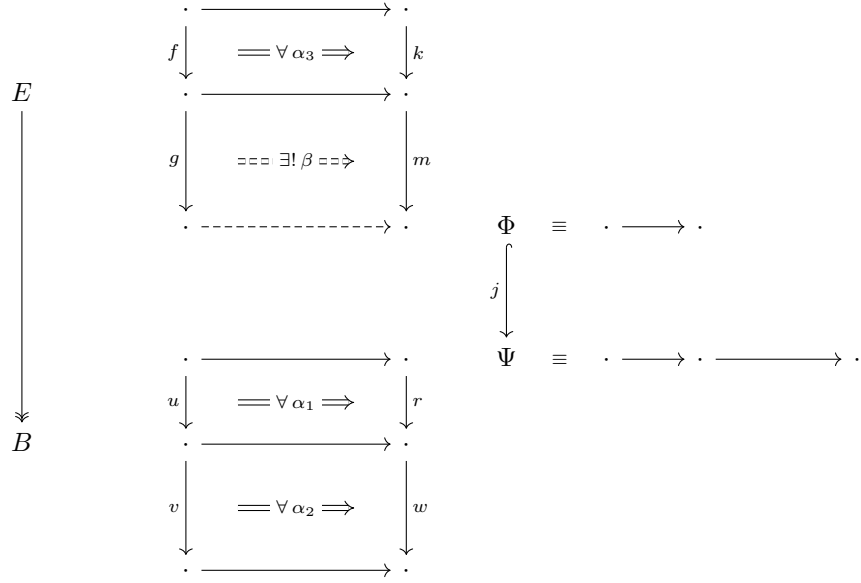
\begin{figure}
	\begin{tikzcd}
		&& \cdot &&& \cdot \\
		E && \cdot &&& \cdot \\
		\\
		&& \cdot &&& \cdot & \Phi & \cdot & \cdot \\
		\\
		&& \cdot &&& \cdot & \Psi & \cdot & \cdot && \cdot \\
		B && \cdot &&& \cdot \\
		\\
		&& \cdot &&& \cdot
		\arrow[from=1-3, to=1-6]
		\arrow[""{name=0, anchor=center, inner sep=0}, "f"', from=1-3, to=2-3]
		\arrow[from=2-3, to=2-6]
		\arrow[""{name=1, anchor=center, inner sep=0}, "k", from=1-6, to=2-6]
		\arrow[""{name=2, anchor=center, inner sep=0}, "g"', from=2-3, to=4-3]
		\arrow[dashed, from=4-3, to=4-6]
		\arrow[""{name=3, anchor=center, inner sep=0}, "m", from=2-6, to=4-6]
		\arrow[""{name=4, anchor=center, inner sep=0}, "u"', from=6-3, to=7-3]
		\arrow[from=7-3, to=7-6]
		\arrow[from=6-3, to=6-6]
		\arrow[""{name=5, anchor=center, inner sep=0}, "r", from=6-6, to=7-6]
		\arrow[""{name=6, anchor=center, inner sep=0}, "v"', from=7-3, to=9-3]
		\arrow[from=9-3, to=9-6]
		\arrow[""{name=7, anchor=center, inner sep=0}, "w", from=7-6, to=9-6]
		\arrow["j"', hook, from=4-7, to=6-7]
		\arrow[from=4-8, to=4-9]
		\arrow[from=6-8, to=6-9]
		\arrow[from=6-9, to=6-11]
		\arrow[two heads, from=2-1, to=7-1]
		\arrow["\jdeq"{description}, draw=none, from=4-7, to=4-8]
		\arrow["\jdeq"{description}, draw=none, from=6-7, to=6-8]
		\arrow["{\exists! \, \beta}"{description}, shorten <=19pt, shorten >=19pt, Rightarrow, dashed, from=2, to=3]
		\arrow["{\forall\, \alpha_3}"{description}, shorten <=19pt, shorten >=19pt, Rightarrow, from=0, to=1]
		\arrow["{\forall \, \alpha_2}"{description}, shorten <=19pt, shorten >=19pt, Rightarrow, from=6, to=7]
		\arrow["{\forall \, \alpha_1}"{description}, shorten <=19pt, shorten >=19pt, Rightarrow, from=4, to=5]
	\end{tikzcd}
	\caption{Universal property of $j$-LARI cells (schematic illustration)}
	\label{fig:lari-cell}
\end{figure}

\subsection{LARI fibrations}\label{ssec:lari-fibs}

\begin{definition}[$j$-LARI family]\label{def:enough-lari-lifts}
	Let $P:B \to \UU$ be an isoinner family over a Rezk type $B$, and $j:\Phi \hookrightarrow \Psi$ a shape inclusion $j$-LARI lift. We call $P$ a \emph{$j$-LARI family} if $P$ has \emph{enough $j$-LARI lifts}, meaning the following type is inhabited:
	\[ \prod_{u:\Phi \to B} \prod_{v:\ndexten{\Psi}{B}{\Phi}{u}} \prod_{f:\prod_{t:\Phi} P(u(t))} \sum_{g:\exten{t:\Psi}{P(v(t))}{\Phi}{f}} \mathrm{isLariCell}^P_j(g) \] 
\end{definition}

By~\Cref{thm:reladj-char}, in fact $\sum_{g:\ldots} \isLARICell_{j}(g)$ is a proposition. Given $\angled{u,v,f}$, we denote the arrow part from the center of contraction of this type, occuring in~(\ref{eq:enough-lari-cell}), as
\[ P_!(u,v,f) \defeq g_{u,v,f} :\exten{\Psi}{v^*P}{\Phi}{f}, \]
generalizing from cocartesian families.\footnote{We could also add the inclusion $j:\Phi \hookrightarrow \Psi$ as an annotation, but this is not necessary here since we will only deal with one such inclusion at a time.}
Similarly, for $g\jdeq P_!(u,v,f)$ we denote the ensuing ``filling'' data from~\Cref{eq:lari-cell-explicit} by
\[ \tyfill_g(\alpha,m) : \exten{\pair{t}{s}:\Delta^1 \times \Psi} {P\big(\alpha_1(t,s)\big)} {b_1 \poprod j}{[\pair{g}{m},\alpha_2]}. \]

The following generalizes the classical \emph{Chevalley condition}~\cite{StrYon,GrayFib,RV21} to arbitrary shape inclusions.
\begin{theorem}[$j$-LARI families via enough $j$-LARI lifts]\label{thm:lari-fams-lifting}
	Let $B$ be a Rezk type, $P: B \to \UU$ be an isoinner family, and denote by $\pi: E \to B$ the associated projection map. Then $P$ has enough $j$-LARI lifts if and only if it is a $j$-LARI family,~\ie~the Leibniz cotensor map $\pi' \defeq i_0 \cotens \pi: E^\Psi \to E^\Phi \to_{B^\Phi} B^\Psi$ has a left adjoint right inverse:
	\[\begin{tikzcd}
		{E^{\Psi}} & {} \\
		&& {E^\Phi \times_{B^\Phi} B^\Psi} && {E^\Phi} \\
		&& {B^{\Psi}} && {B^\Phi}
		\arrow[two heads, from=2-3, to=3-3]
		\arrow["{B^j}"', from=3-3, to=3-5]
		\arrow[from=2-3, to=2-5]
		\arrow["{\pi^\Phi}", two heads, from=2-5, to=3-5]
		\arrow["\lrcorner"{anchor=center, pos=0.125}, draw=none, from=2-3, to=3-5]
		\arrow["{E^j}", shift left=2, curve={height=-18pt}, from=1-1, to=2-5]
		\arrow[""{name=0, anchor=center, inner sep=0}, "\chi"', curve={height=12pt}, dotted, from=2-3, to=1-1]
		\arrow[""{name=1, anchor=center, inner sep=0}, "{\pi'}"', curve={height=12pt}, from=1-1, to=2-3]
		\arrow["{\pi^\Psi}"', shift right=2, curve={height=18pt}, two heads, from=1-1, to=3-3]
		\arrow["\dashv"{anchor=center, rotate=-117}, draw=none, from=0, to=1]
	\end{tikzcd}\]
\end{theorem}

\begin{proof}
	Assume $P:B \to \UU$ is an isoinner family with enough $j$-LARI lifts. The gap map can be taken as the strict projection
	\[ \pi' \defeq \lambda u,v,f,g.\angled{u,v,f}: E^{\Psi} \to E^\Phi \times_{B^\Phi} B^\Psi.\]
	For the candidate LARI we take the map that produces the $j$-LARI lift, \ie
	\[ \chi \defeq \lambda u,v,f.\angled{u,v,f,P_!(u,v,f)} : E^\Phi \times_{B^\Phi} B^\Psi \to E^\Psi.\]
	This is by definition a (strict) section of $\pi'$.
	
	For $\angled{u,v,f}:E^\Phi \times_{B^\Phi} B^\Psi$ and $\angled{r,w,k,m} : E^{\Psi}$ we define the maps
	\[\begin{tikzcd}
		{\hom_{E^\Psi}(\chi(u,v,f), \angled{r,w,k,m})} && {\hom_{E^\Phi \times_{B^\Phi} B^\Psi}(\angled{u,v,f},\angled{r,w,k})}.
		\arrow["F_{\angled{u,v,f},\angled{r,w,k,m}}", shift left=2, from=1-1, to=1-3]
		\arrow["G_{\angled{u,v,f},\angled{r,w,k,m}}", shift left=2, from=1-3, to=1-1]
	\end{tikzcd}\]
	defined by\footnote{We decompose morphisms in $E^\Psi$ as pairs $\pair{\alpha}{\beta}$ where $\alpha$ denotes the part in $E^\Phi \times_{B^\Phi} B^\Psi$ and $\beta$ is the given $\Psi$-shaped cell in $P$ lying over.}
	\[ F_{\angled{u,v,f},\angled{r,w,k,m}}(\alpha,\beta) \defeq \alpha, \quad G_{\angled{u,v,f},\angled{r,w,k,m}}(\gamma) \defeq \pair{\gamma}{\tyfill_{P_!(u,v,f)}(\gamma,m)} \]
	are quasi-inverse to one another.\footnote{For brevity, we shall henceforth leave the fixed parameters $\angled{u,v,f},\angled{r,w,k,m}$ implicit.}
	
	Clearly, $G$ is a section of the projection $F$ since for a morphism $\gamma$ we find
	\[ F(G(\gamma)) = F(\gamma,\tyfill_{P_!(u,v,f)}(\gamma,m)) = \gamma.\]
	
	For a morphism $\pair{\alpha}{\beta}$ in the opposing transposing morphism space we obtain
	\[ G(F(\alpha,\beta)) = G(\alpha) = \pair{\alpha}{\tyfill_{P_!(u,v,f)}(\alpha,m)}\]
	where we obtain an identification $\beta = \tyfill_{P_!(u,v,f)}(\alpha,m)$ (over $\refl_\alpha$) because of the universal property~\Cref{eq:lari-cell-explicit}. 
	
	This suffices to show that $\chi \dashv \pi'$ is a LARI adjunction as claimed.
	
	Conversely, suppose $\chi$ is a given LARI of $\pi'$, \wlogg~a strict section. This gives, for any data $\angled{u,v,f}:E^\Phi \times_{B^\Phi} B^\Psi$ we obtain a (strictly) commutative triangle:
	\[\begin{tikzcd}
		&& {E^\Psi} \\
		\unit && {B^\Psi \times_{B^\Phi} E^\Psi}
		\arrow["{\chi(u,v,f)}", from=2-1, to=1-3]
		\arrow[""{name=0, anchor=center, inner sep=0}, "{\angled{u,v,f}}"', from=2-1, to=2-3]
		\arrow["{\pi' }", from=1-3, to=2-3]
		\arrow[shorten <=15pt, shorten >=15pt, Rightarrow, no head, from=1-3, to=0]
	\end{tikzcd}\]
	
	Moreover, for all $\angled{r,w,k,m} : E^\Psi$, the map
	\[\begin{tikzcd}
		{\hom_{E^\Psi}(\chi(u,v,f),\langle r,w,k,m\rangle)} && {\hom_{E^\Phi \times_{B^\Phi} B^\Psi}(\langle u,v,f\rangle,\langle r,w,k\rangle)}
		\arrow["F_{\angled{r,w,k,m}}", from=1-1, to=1-3]
	\end{tikzcd}\]
	defined by
	\[ F(\alpha,\beta) \jdeq \alpha\]
	is an equivalence. Finally, contractibility of the fibers amounts to the universal property~\Cref{eq:lari-cell-explicit}, but this precisely means that $\chi(u,v,f)$ is a $j$-LARI cell.
\end{proof}

\subsection{LARI functors}\label{ssec:lari-fun}

Let $j:\Phi \hookrightarrow \Psi$ be a shape inclusion.
\begin{definition}[$j$-LARI functors]\label{def:lari-fun}
	Over Rezk types $A$ and $B$, resp., consider $j$-LARI families $Q:A \to \UU$ and $P:E\to \UU$, resp., with
	\[ \xi \defeq \Un_A(Q): F \fibarr A, \quad \pi \defeq \Un_E(P): E \fibarr B. \]
	Assume there is a fibered functor from $\xi$ to $\pi$ given by a commutative square:
	\[\begin{tikzcd}
		F && E \\
		A && B
		\arrow["\xi"', two heads, from=1-1, to=2-1]
		\arrow["\alpha"', from=2-1, to=2-3]
		\arrow["\pi", two heads, from=1-3, to=2-3]
		\arrow["\varphi", from=1-1, to=1-3]
	\end{tikzcd}\]
	This defines a \emph{$j$-LARI functor} if and only if the following proposition is satisfied:
	\[ \prod_{m:\Psi \to F} \isLARICell_{j}^Q(m) \to \isLARICell_{j}^P(\varphi\,m)\]
\end{definition}

\begin{proposition}[Naturality of $j$-LARI functors]\label{prop:nat-larifun}
	Let $B$ be a Rezk type, and $P:B \to \UU$, $Q:C \to \UU$, resp. be $j$-LARI families with associated fibrations $\xi:F \fibarr A$ and $\xi:E \fibarr B$, resp. Then the proposition that $\pair{\alpha}{\varphi}$ be a $j$-LARI functor from $\xi$ to $\pi$ is logically equivalent to $\pair{\alpha}{\varphi}$ commuting with cocartesian lifts: \ie~for any $\angled{r,w,k,m} : F^\Phi \times_{A^\Phi} A^\Psi$ there exists an identification\footnote{suppressing the ``lower'' data which can be taken to consist of identities anyway} of $\Psi$-cells
	\[ \varphi\big(Q_!(r,w,k)\big) = P_!(\alpha \,r, \alpha\,w, \varphi\,k, \varphi\,d). \]
\end{proposition}

\begin{proof}
	Since $\varphi$ is a $j$-LARI functor by assumption $g \defeq \varphi\big(Q_!(r,w,k)\big)$ is a $j$-LARi cell,~\ie~$\relAdj{g}{\angled{u,v,f}}{\pi'}$. But also for the cell $g' \defeq P_!(\alpha \,r, \alpha\,w, \varphi\,k, \varphi d)$ by construction we have \\ $\relAdj{g'}{\angled{u,v,f}}{\pi'}$, hence there is a homotopy $g=g'$, by uniqueness of relative left adjoints.
	
	Conversely, since any $j$-LARI arrow in a family occurs as a $j$-LARI lift (of the data given by projection and restriction), the assumed identifications yield the desired implication.
\end{proof}

\begin{theorem}[Chevalley criterion for $j$-LARI functors, \cf~\protect{\cite[Theorem~5.3.19]{BW21}}, {\protect\cite[Theorem~5.3.4]{RV21}}]\label{thm:char-lari-fun}
	Given data as in~\Cref{def:lari-fun}, the following are equivalent:
	\begin{enumerate}
		\item The fiberwise map $\pair{\alpha}{\varphi}$ is a $j$-LARI functor.
		\item The mate of the induced canonical natural isomorphism is invertible, too:
		\[\begin{tikzcd}[sep=4ex]
			{F^{\Psi}} && {E^\Psi} & {F^\Psi} && {E^{\Delta^1}} & {} \\
			{F^\Phi \times_{A^\Phi} A^\Psi} && {E^\Phi \times_{B^\Phi} B^\Psi} & {F^\Phi \times_{A^\Phi} A^\Psi} && {E^\Phi \times_{B^\Phi} B^\Psi} & {}
			\arrow["{\xi'}"', from=1-1, to=2-1]
			\arrow["{\varphi^\Phi \times_{\alpha^\Phi} \alpha^\Psi}"', from=2-1, to=2-3]
			\arrow["{\varphi^{\Delta^1}}", from=1-1, to=1-3]
			\arrow[""{name=0, anchor=center, inner sep=0}, "{r'}", from=1-3, to=2-3]
			\arrow["{=}", shorten <=9pt, shorten >=9pt, Rightarrow, from=2-1, to=1-3]
			\arrow[""{name=1, anchor=center, inner sep=0}, "\ell", from=2-4, to=1-4]
			\arrow["{\varphi^\Phi \times_{\alpha^\Phi} \alpha^\Psi}"', from=2-4, to=2-6]
			\arrow["{\varphi^{\Delta^1}}", from=1-4, to=1-6]
			\arrow["{\ell'}"', from=2-6, to=1-6]
			\arrow["{=}"', shorten <=12pt, shorten >=23pt, Rightarrow, from=2-6, to=1-4]
			\arrow["\rightsquigarrow", Rightarrow, draw=none, from=0, to=1]
		\end{tikzcd}\]
	\end{enumerate}
\end{theorem}

\begin{proof}
	The counit of the adjunction exhibiting $\pi:E \fibarr B$ as a $j$-LARI fibration, at stage $\angled{u,v,f,g}:E^\Psi$, can be taken to be
	\begin{align}\label{eq:counit-lari-fib}
		\varepsilon_{u,v,f,g} \jdeq \angled{\id_u,\id_v,\id_f,\tyfill_{P_!(u,v,f)}(\id_u,\id_v,\id_f,g)},
	\end{align}
	as one sees by the usual construction from the transposing map, \cf~the proof of~\Cref{thm:lari-fams-lifting}. Now, from the proof of~\cite[Proposition~A.1.2]{BW21} \footnote{\cf~\cite[Theorem~5.3.19]{BW21} for the cocartesian case} we see that the pasting cell constructed from the diagram
	\[\begin{tikzcd}
		{F^\Phi \times_{A^\Phi} A^\Psi} & {F^\Psi} && {E^\Psi} \\
		& {F^\Phi \times_{A^\Phi} A^\Psi} && {E^\Phi \times_{B^\Phi} B^\Psi} && {E^\Psi}
		\arrow["\mu", from=1-1, to=1-2]
		\arrow["{\xi'}", from=1-2, to=2-2]
		\arrow[""{name=0, anchor=center, inner sep=0}, Rightarrow, no head, from=1-1, to=2-2]
		\arrow["{\varphi^\Phi \times_{\alpha^\Phi} \alpha^\Psi}"', from=2-2, to=2-4]
		\arrow["{\varphi^\Psi}", from=1-2, to=1-4]
		\arrow["{\pi'}"', from=1-4, to=2-4]
		\arrow[shorten <=18pt, shorten >=18pt, Rightarrow, no head, from=2-2, to=1-4]
		\arrow["\chi"', from=2-4, to=2-6]
		\arrow[""{name=1, anchor=center, inner sep=0}, Rightarrow, no head, from=1-4, to=2-6]
		\arrow[shorten <=4pt, shorten >=4pt, Rightarrow, no head, from=0, to=1-2]
		\arrow[shorten <=4pt, shorten >=4pt, Rightarrow, no head, from=2-4, to=1]
	\end{tikzcd}\]
	is given, at stage $\angled{r,w,k,m}:F^\Psi$ by
	\begin{align*} 
		& \widetilde{\varepsilon}_{r,w,k,m} = \varepsilon_{\alpha\,r,\alpha\,w,\varphi\,k,\varphi(Q_!(r,w,k))} \\
		& = \angled{\id_{\alpha\,r},\id_{\alpha\,w},\id_{\varphi\,k},\tyfill_{P_!(\alpha\,r,\alpha\,w,\alpha\,k)}(\id_{\alpha\,r},\id_{\alpha\,w},\id_{\varphi\,k},\varphi Q_!(r,w,k))}.
	\end{align*}
	This collapses to just the comparison map $P_!(\alpha\,r,\alpha\,w,\alpha\,k) \to \varphi Q_!(r,w,k)$ given by filling. Now by~\Cref{prop:nat-larifun} the invertibility of this is equivalent to $\pair{\alpha}{\varphi}$ being a $j$-LARI functor. 
\end{proof}

\section{Cocartesian arrows, fibrations, and functors}

Again, we fix the following. Let $P : B \to \UU$ be an isoinner family over a Rezk type $B$, with associated isoinner fibration $\pi \jdeq \pr_1 : E \fibarr B$. In this section, we want to show the notions of $j$-LARI cells, fibrations, and functors from~\ref{sec:lari-stuff} specialize to cocartesian arrows, fibrations, and functors, by taking for the shape inclusion $j$ the inclusion of the initial point $0$ into $\Delta^1$.

Part of the characterizations come out as corollaries from our general theorems in~\Cref{sec:lari-stuff}, namely when the involve ALLD, or, resp., LARI conditions. In the case for relative, or, resp., fibered adjoint conditions, we give a direct proof since, in the framework of~\Cref{sec:lari-stuff} we cannot generalize the structure needed for these kinds of conditions.

\subsection{Cocartesian arrows}

Cocartesian arrows are dependent arrows that satisfy an initiality property which we will exhibit as an instance of~\Cref{def:lari-cell}. They are those arrows that characterize the lifting properties of cocartesian families, which in turn encode fibered $\infty$-categories.

\begin{definition}[Cocartesian arrows]
	Consider the isoinner fibration $\pi : E \fibarr B$. We call $f : \Delta^1 \to E$ a \emph{cocartesian arrow} if it is an $i_0$-LARI cell, where $i_0 : \unit \to \Delta^1$ is the inclusion of $\{0\}$ into $\Delta^1$.

	This means, the following diagram is an absolute left lifting diagram:
\[\begin{tikzcd}
	&& {E^{\Delta^1}} \\
	\unit && {B^{\Delta^1} \times_B E}
	\arrow[""{name=0, anchor=center, inner sep=0}, "{\pair{u}{e}}"', from=2-1, to=2-3]
	\arrow["{i_0 \cotens \pi}", from=1-3, to=2-3]
	\arrow["\pair{b}{e}", from=2-1, to=1-3]
	\arrow[shorten <=10pt, shorten >=14pt, Rightarrow, no head, from=1-3, to=0]
\end{tikzcd}\]
	which induces a fibered equivalence 
	\[ \comma{f}{E^{\Delta^1}} \simeq_{E^{\Delta^1}} \comma{\pair{u}{e}}{\pi'}  \]
	By the discussion around~\Cref{def:lari-cell}, this is equivalent to:
	\[ \prod_{b,b',b'',b''' : B} \prod_{\substack{e:P\,b, \\ e' : P\,b', \\ e'' : P\,b'', \\ e''' : P\,b'''}} 
	  \prod_{\substack{u:b \to_B b' \\ v:b' \to_B b'' \\ w:b \to_B b'' \\ w:b'' \to_B b'''}}
	  \prod_{\sigma : vu =_B w'w}  \prod_{\substack{f:e \to^P_u e' \\ h : e \to^P_w e'' \\ h' : e'' \to^P_{w'} e'''}} \isContr\left( \sum_{g:e' \to^P e'''} gf =^P_\sigma h'h \right)\]
\end{definition}

This is tantamount to a ``cubical version'' of the cocartesianness condition, which is a ``degenerate'' instance of the general picture for $j$-LARI cells, cf.~\Cref{fig:lari-cell}: 
\[\begin{tikzcd}
	e & {e''} \\
	{e'} & {e'''} \\
	b & {b''} \\
	{b'} & {b'''}
	\arrow["u"', from=3-1, to=4-1]
	\arrow["w", from=3-1, to=3-2]
	\arrow["{w'}", from=3-2, to=4-2]
	\arrow["v", from=4-1, to=4-2]
	\arrow["f"', from=1-1, to=2-1]
	\arrow["h", from=1-1, to=1-2]
	\arrow["{h'}", from=1-2, to=2-2]
	\arrow["g", dashed, from=2-1, to=2-2]
\end{tikzcd}\]

We will use the notation $\pi \downarrow B \defeq B^{\Delta^1} \times_B E$, see also the discussion of comma types~\cite[Section~2.6]{BW21} and comma fibrations~\cite[Subsection~4.3]{W22-2sCart}.

\begin{theorem}[Characterization of cocartesian arrows]\label{thm:cocart-arr-char}
	Let $P:B \to \UU$ be an isoinner family over a Rezk type $B$. For $u:b \to_B b'$, consider a dependent arrow $f: e \to^P_u e'$ with $e:P\,b$, $e':P\,b'$. Then the following conditions are equivalent: $f$ is a $P$-cocartesian arrow if and only if either of :
	\begin{enumerate}
		\item\label{it:cocart-arr-char-i} The dependent arrow $f: e \to^P_u e'$ is $P$-cocartesian.
		\item\label{it:cocart-arr-char-ii} The diagram
	\[\begin{tikzcd}
	&& E \\
	\unit && {\comma{\pi}{B}}
	\arrow[""{name=0, anchor=center, inner sep=0}, "{{\pair{u}{e}}}"', from=2-1, to=2-3]
	\arrow["{\pi'}", two heads, from=1-3, to=2-3]
	\arrow["{{f}}", from=2-1, to=1-3]
	\arrow[shorten <=11pt, shorten >=11pt, Rightarrow, no head, from=0, to=1-3]
	\end{tikzcd}\]
		is an absolute left lifting diagram.
		\item\label{it:cocart-arr-char-iii} The morphism $\sigma \defeq \langle u, \id_{b'}, f \rangle: \pair{u}{e} \to \iota \, \pair{b}{e}$ seen as a $2$-cell
		\[\begin{tikzcd}
			&& E \\
			\unit && {\comma{\pi}{B}}
			\arrow[""{name=0, anchor=center, inner sep=0}, "{\pair{u}{e}}"', from=2-1, to=2-3]
			\arrow["\iota", two heads, from=1-3, to=2-3]
			\arrow["{\pair{b}{e}}", from=2-1, to=1-3]
			\arrow["\sigma"{description}, shorten <=8pt, shorten >=8pt, Rightarrow, from=0, to=1-3]
		\end{tikzcd}\]
		defines the unit of a relative adjunction
		\[  \comma{\pair{b}{e}}{E} \simeq \comma{\pair{u}{e}}{\iota}.\]
	\end{enumerate}
\end{theorem}

\begin{proof}

The essential point is to show that the standard cocartesian lifting condition
\begin{align}\label{eq:cocart-lift}
	 \prod_{b,b',b'' : B} \prod_{\substack{e:P\,b, \\ e' : P\,b', \\ e'' : P\,b''}} 
	\prod_{\substack{u:b \to_B b' \\ v:b' \to_B b'' }}
	\prod_{\kappa : v \circ u=_B vu}
	 \prod_{\substack{f:e \to^P_u e' \\ h : e \to^P_{vu} e''}} \isContr\left( \sum_{g:e' \to^P e'''} gf =^P_\kappa h'h \right)
	\end{align}
is equivalent to the ``cubical'' cocartesian lifting condition
\begin{align}\label{eq:cocart-lift-cubical}
 \prod_{b,b',b'',b''' : B} \prod_{\substack{e:P\,b, \\ e' : P\,b', \\ e'' : P\,b'', \\ e''' : P\,b'''}} 
\prod_{\substack{u:b \to_B b' \\ v:b' \to_B b'' \\ w:b \to_B b'' \\ w':b'' \to_B b'''}}
\prod_{\sigma : vu =_B w'w}  \prod_{\substack{f:e \to^P_u e' \\ h : e \to^P_w e'' \\ h' : e'' \to^P_{w'} e'''}} \isContr\left( \sum_{g:e' \to^P e'''} gf =^P_\sigma h'h \right).
\end{align}
Proposition~(\ref{eq:cocart-lift}) implies Proposition~(\ref{eq:cocart-lift-cubical}) as follows: assume given data
\[ \langle b,b',b'',e,e',e'',e''',u,v,w,w',\sigma,f,h,h'\rangle\]as signified in Proposition~(\ref{eq:cocart-lift-cubical}). Then, using the identification (\cf~\cite[Subsection~3.4]{RS17})
\begin{align}\label{eq:square}
	\Delta^1 \times \Delta^1 \simeq \Delta^2 \cup_{\Delta_1^1} \Delta^2,
\end{align} instantiating Proposition~(\ref{eq:cocart-lift}) with $\langle b,b',b''',e,e',e''', u,v, \sigma':vu = w'w, f, h'h \rangle$ we obtain, up to homotopy, a unique pair $\pair{g:e' \to^P_v e'''}{\tau':g \circ f =^P_{\sigma'} h'h}$. Via (fibered) transport along the equivalence (\ref{eq:square}), we obtain an inhabitant of Proposition~(\ref{eq:cocart-lift-cubical}).

For the direction from Proposition~(\ref{eq:cocart-lift-cubical}) to Proposition~(\ref{eq:cocart-lift}), we can instantiate the square by so that $h$ or $h'$ become identities. In sum, this shows that the propositions (\ref{eq:cocart-lift}) and (\ref{eq:cocart-lift-cubical}) are equivalent. 

Now, Conditions~\ref{it:cocart-arr-char-i}~and~\ref{it:cocart-arr-char-ii} are equivalent by definition. Instantiating $j : \Phi \hookrightarrow \Psi$ in~\Cref{ssec:lari-cells} as $i_0 : \unit \hookrightarrow \Delta^1$ one gets that Condition~\ref{it:cocart-arr-char-ii} is equivalent to Proposition~(\ref{eq:cocart-lift-cubical}).

Now, Condition~\ref{it:cocart-arr-char-iii} says that the $2$-cell $\langle u, \id_{b'}, f \rangle$ gives rise to a relative adjunction. By~\Cref{thm:reladj-char} this means that the map
 \[ \Psi :  \underbrace{\prod_{\pair{b''}{e''}:E} \big( \pair{b'}{e'} \to_E \pair{b''}{e''} \big) \to \big( \pair{u}{e} \to_{\comma{\pi}{B}} \pair{\id_{b''}}{e''} \big)}_{\simeq (\comma{\pair{b}{e}}{E} \to \comma{\pair{u}{e}}{\iota})} \]
 defined by
 \[ \Psi_{b'',e''}(v:b' \to b'', g:e'\to_v e'') \defeq \langle \langle vu, \id_v  \rangle : u \rightrightarrows \id_{b''}, gf: e \to_{vu} e''\rangle \]
 is an equivalence. Using manipulations of extension types (see~\cite[Section~4.4]{RS17} and \cite[Subsubsection~2.1.2 and Subsection~2.4]{BW21}) this proposition is seen to be equivalent to Condition~\ref{eq:cocart-lift-cubical}. 
\end{proof}

As discussed in~\cite[Section~5.1]{BW21}, concretely for the case of cocartesian arrows, there exist yet more ways to characterize them, to be found in the classical literature such as~\cite{JoyQcat} and~\cite{LurHTT}.

\subsection{Cocartesian families}

Cocartesian families, or equivalently, cocartesian fibrations describe families of $\infty$-categories parametrized by an $\infty$-category. In particular, these families are \emph{functorial}, in that we can lift arrows from the base to the family, in such a way that they compose and lift identities to identities (all up to homotopy).

An extensive treatment of cocartesian families in Riehl--Shulman's simplicial homotopy type theory~\cite{RS17} is given in~\cite{BW21}.

\begin{definition}[Cocartesian family]
	Let $P : B \to \UU$ be an isoinner family over a Rezk type $B$. We call $P$ a \emph{cocartesian family} if it is an $i_0$-LARI family.
\end{definition}

\begin{theorem}[Characterization of cocartesian families]
	Let $P : B \to \UU$ be an isoinner family over a Rezk type $B$, with associated projection map $\pi : E \fibarr B$. Then the following are equivalent:
	\begin{enumerate}
		\item\label{it:char-cocart-fam-i} The family $P$ is cocartesian.
		\item\label{it:char-cocart-fam-ii} The Leibniz cotensor map $i_0 \cotens \pi: E^{\Delta^1} \to \comma{\pi}{B}$ has a left adjoint right inverse, the \emph{cocartesian lifting map}:
		\[\begin{tikzcd}
			{E^{\Delta^1}} & {} \\
			&& {\pi \downarrow B} && E \\
			&& {B^{\Delta^1}} && B
			\arrow[two heads, from=2-3, to=3-3]
			\arrow["{\partial_0}"', from=3-3, to=3-5]
			\arrow[from=2-3, to=2-5]
			\arrow["\pi", two heads, from=2-5, to=3-5]
			\arrow["\lrcorner"{anchor=center, pos=0.125}, draw=none, from=2-3, to=3-5]
			\arrow["{\partial_0}", shift left=2, curve={height=-18pt}, from=1-1, to=2-5]
			\arrow[""{name=0, anchor=center, inner sep=0}, "\chi"', curve={height=12pt}, dashed, from=2-3, to=1-1]
			\arrow[""{name=1, anchor=center, inner sep=0}, "{i_0\widehat{\pitchfork}\pi}"', curve={height=12pt}, from=1-1, to=2-3]
			\arrow["{\pi^{\Delta^1}}"', shift right=2, curve={height=18pt}, two heads, from=1-1, to=3-3]
			\arrow["\dashv"{anchor=center, rotate=-118}, draw=none, from=0, to=1]
		\end{tikzcd}\]
		\item\label{it:char-cocart-fam-iii} The map
		\[ \iota \defeq \iota_P : E \to \comma{\pi}{B}, \quad \iota \, \pair{b}{e} :\jdeq \pair{\id_b}{e}  \]
		has a fibered left adjoint $\tau \defeq \tau_P: \comma{\pi}{B} \to E$, the \emph{cocartesian transport map}, as indicated in the diagram:
		\[\begin{tikzcd}
			{} && E && {\pi \downarrow B} \\
			\\
			&&& B
			\arrow[""{name=0, anchor=center, inner sep=0}, "\iota"', curve={height=12pt}, from=1-3, to=1-5]
			\arrow["\pi"', two heads, from=1-3, to=3-4]
			\arrow["{\partial_1'}", two heads, from=1-5, to=3-4]
			\arrow[""{name=1, anchor=center, inner sep=0}, "\tau"', curve={height=12pt}, from=1-5, to=1-3]
			\arrow["\dashv"{anchor=center, rotate=-90}, draw=none, from=1, to=0]
		\end{tikzcd}\]
	\end{enumerate}
\end{theorem}

\begin{proof}
	The equivalence of Items~(\ref{it:char-cocart-fam-i}) and (\ref{it:char-cocart-fam-ii}) follows by~\Cref{thm:lari-fams-lifting}. The equivalence of~(\ref{it:char-cocart-fam-iii}) with the other two is proven in~\cite[Theorem~5.2.7]{BW21}.
\end{proof}

\subsection{Cocartesian functors}

Given cocartesian fibrations over a Rezk base $B$, a \emph{cocartesian functor} is a fiberwise map as in
\[\begin{tikzcd}
	F && E \\
	& B
	\arrow["\varphi", from=1-1, to=1-3]
	\arrow["\pi"', two heads, from=1-1, to=2-2]
	\arrow["\xi", two heads, from=1-3, to=2-2]
\end{tikzcd}\]
which maps cocartesian arrows to cocartesian arrows. Classically, these make up the morphisms of the slice-$\infty$-category of cocartesian fibrations over $B$.

We will recover this notion, as expected, as a special case of~\Cref{def:lari-fun}.

\begin{definition}[Cocartesian functor]
	Let $Q : A \to \UU$ and $P : B \to \UU$ be cocartesian families over Rezk types $A$ and $B$. A fiberwise map
	\[ \left \langle j:A \to B, \varphi : \prod_{a:A} Q(a) \to P(j\,a) \right \rangle \]
	is \emph{cocartesian} if it is an $i_0$-LARI functor.
\end{definition}

\subsubsection{Characterizations of cocartesian functors}

\begin{proposition}[Naturality of cocartesian functors]\label{prop:nat-cocartlift-arr}
	Let $A,B$ be Rezk types, $P:A \to \UU$, $Q:B \to \UU$ cocartesian families, and $\Phi \jdeq \pair{j}{\varphi}$ a cocartesian functor. Then $\Phi$ commutes with cocartesian lifts, \ie,~for any $u:\hom_B(a,b)$ there is an identification of arrows
	\[ \varphi \big(P_!(u,d)\big) =_{\Delta^1 \to (ju)^*Q} Q_!(ju,\varphi_ad) \]
	and hence of endpoints
	\[ \varphi_b(u_!^Pd) =_{Q(jb)} (ju)_!^Q(\varphi_ad). \]
	In particular there is a homotopy commutative square:
		\[
	\begin{tikzcd}
		P\,a \ar[rr, "\varphi_a"] \ar[d, "\coliftptfammap{P}{u}" swap] & & Q\,ja \ar[d, "\coliftptfammap{Q}{(ju)}"] \\
		P\,b \ar[rr, "\varphi_b" swap] && Q\,jb
	\end{tikzcd}
	\]
	\end{proposition}

	\begin{proof}
		This follows since it is an instance of~\Cref{prop:nat-larifun}.
	\end{proof}

\begin{theorem}[{\protect\cite[Theorem 5.3.4]{RV21}}]\label{thm:cocart-fun-intl-char}
Let $A$ and $B$ be Rezk types, and consider cocartesian families $P:B \to \UU$ and $Q:A \to \UU$ with associated fibrations $\xi: F \fibarr  A$ and $\pi: E \fibarr B$, resp .

For a fibered functor $\Phi\defeq \pair{j}{\varphi}$ giving rise to a square
\[
	\begin{tikzcd}
		F \ar[r, "\varphi"] \ar[d, "\xi" swap] & E \ar[d, "\pi"] \\
		A \ar[r, "j" swap] & B
	\end{tikzcd}
\]
the following are equivalent:
\begin{enumerate}
	\item\label{it:cocart-fun-cocart} The fiberwise map $\Phi$ is a cocartesian functor.
	
	\item\label{it:cocart-fun-mate} The mate of the induced natural isomorphism is invertible, too:
\[\begin{tikzcd}
	{F^{\Delta^1}} && {E^{\Delta^1}} & {} & {F^{\Delta^1}} && {E^{\Delta^1}} \\
	{\xi \downarrow A} && {\pi \downarrow B} & {} & {\xi \downarrow A} && {\pi \downarrow B}
	\arrow["{r}"', from=1-1, to=2-1]
	\arrow["{\varphi \downarrow j}"', from=2-1, to=2-3]
	\arrow["{\varphi^{\Delta^1}}", from=1-1, to=1-3]
	\arrow["{r'}", from=1-3, to=2-3]
	\arrow["{\rightsquigarrow}" description, from=1-4, to=2-4, phantom, no head]
	\arrow[Rightarrow, "{=}", from=2-1, to=1-3, shorten <=7pt, shorten >=7pt]
	\arrow["{\ell}", from=2-5, to=1-5]
	\arrow["{\varphi \downarrow j}"', from=2-5, to=2-7]
	\arrow["{\varphi^{\Delta^1}}", from=1-5, to=1-7]
	\arrow["{\ell'}"', from=2-7, to=1-7]
	\arrow[Rightarrow, "{=}"', from=2-7, to=1-5, shorten <=7pt, shorten >=7pt]
\end{tikzcd}\]
	\item\label{it:cocart-fun-mate-fibered} The mate of the induced natural isomorphism, fibered over $j:A \to B$, is invertible, too:
	\[\begin{tikzcd}
		{F} && {E} & {} & {F} && {E} \\
		{\xi \downarrow A} && {\pi \downarrow B} & {} & {\xi \downarrow A} && {\pi \downarrow B}
		\arrow["{i}"', from=1-1, to=2-1]
		\arrow["{\varphi}", from=1-1, to=1-3]
		\arrow["{i'}", from=1-3, to=2-3]
		\arrow[Rightarrow, "{=}", from=2-1, to=1-3, shorten <=7pt, shorten >=7pt]
		\arrow["{\rightsquigarrow}" description, from=1-4, to=2-4, phantom, no head]
		\arrow["{\kappa}", from=2-5, to=1-5]
		\arrow["{\varphi}", from=1-5, to=1-7]
		\arrow["{\varphi \downarrow j}"', from=2-5, to=2-7]
		\arrow["{\kappa'}"', from=2-7, to=1-7]
		\arrow[Rightarrow, "{=}"', from=2-7, to=1-5, shorten <=7pt, shorten >=7pt]
		\arrow["{\varphi \downarrow j}"', from=2-1, to=2-3]
	\end{tikzcd}\]

\end{enumerate}

\end{theorem}

\begin{proof}
	The equivalence of Items~(\ref{it:cocart-fun-cocart}) and (\ref{it:cocart-fun-mate}) follows by~\Cref{thm:char-lari-fun}. The equivalence of~(\ref{it:cocart-fun-mate-fibered}) with the other two is proven in~\cite[Theorem~5.3.19]{BW21}.
\end{proof}

\phantomsection%
\printbibliography[heading=bibintoc]

\end{document}